\documentclass[MSNbibl,number,citesort,dvips]{arxbj}

\aid{0}
\volume{20}
\issue{4}
\pubyear{2014}
\firstpage{2020}
\lastpage{2038}
\doi{10.3150/13-BEJ549} 

\makeatletter
\newcommand{\rrVert}{\Vert}
\newcommand{\rrvert}{\vert}
\newcommand{\llVert}{\Vert}
\newcommand{\llvert}{\vert}
\renewcommand{\epsilon}{\varepsilon}
\newcommand{\eqref}[1]{(\ref{#1})}
\renewcommand{\citep}{\cite}
\newcommand{\trup}[2]{{#1}/{#2}}
\newcommand{\mr}{\mathbb{R}}
\newcommand{\mn}{\mathbb{N}}
\newcommand{\me}{\mathbb{E}}
\newcommand{\mpr}{\mathbb{P}}
\newcommand{\mprn}{\mathbb{P}_n}
\newcommand{\undz}{\underline{Z}}
\newcommand{\ovez}{\overline{Z}}

\newtheorem{theorem}{Theorem}[section]
\newtheorem{corollary}{Corollary}[section]
\newremark{remark}{Remark}[section]
\newtheorem{lemma}{Lemma}[section]
\makeatother

\begin{document}
\begin{frontmatter}

\title{New concentration inequalities for suprema of empirical processes}
\runtitle{Concentration of empirical processes}

\begin{aug}
\author{\inits{J.}\fnms{Johannes} \snm{Lederer}\corref{}\thanksref{e1}\ead[label=e1,mark]{johanneslederer@mail.de}} \and
\author{\inits{S.}\fnms{Sara} \snm{van de Geer}\thanksref{e2}\ead[label=e2,mark]{geer@stat.math.ethz.ch}}
\address{Seminar f\"ur Statistik, ETH Z\"urich,
R\"amistrasse 101,
8092 Z\"urich, Switzerland.\\ \printead{e1,e2}}
\end{aug}

\received{\smonth{2} \syear{2012}}
\revised{\smonth{5} \syear{2013}}

%
\begin{abstract}
While effective concentration inequalities for suprema of empirical
processes exist under boundedness or strict tail assumptions, no comparable
results have been available under considerably weaker assumptions. In this
paper, we derive concentration inequalities assuming only low moments for
an envelope of the empirical process. These concentration inequalities
are beneficial even when the envelope is much larger than the single
functions under consideration.
\end{abstract}

%

\end{frontmatter}

\section{Introduction}

Powerful concentration and deviation inequalities for suprema of empirical
processes have been derived during the last 20 years. These inequalities
turned out to be crucial for example, in the study of consistency and rates
of convergence for many estimators. Unfortunately, the known inequalities
are only valid for bounded empirical processes or under strict tail
assumptions. So, this paper was prompted by the question whether useful
inequalities can be obtained under considerably weaker assumptions.

Let us first set the framework, starting with a brief summary of the known
results for bounded empirical processes, or more precisely, for empirical
processes index by bounded functions. To this end, we consider independent
and identically distributed random variables $X_1,\ldots,X_n$ and a countable
function class $\mathcal F$ such that $\sup_{f\in\mathcal{F}}\Vert
f\Vert_\infty\leq1$ and $\sup_{f\in\mathcal{F}}|\me f(X_1)|=0$. The
quantity of interest is denoted by $Y:=\sup_{f\in\mathcal{F}}|\frac
{1}{n}\sum_{i=1}^n
f(X_i)|$ and the
root of the maximal variance by $\sigma:=\sup_{f\in\mathcal
{F}}\sqrt{\me[f(X_1)]^2}$. Refining Rio's
proof in \citep{Rio02} (see also \citep{Massart07}, Chapter~5.3, for the
proof techniques), Bousquet derives
in \citep{Bousquet02} an exponential deviation inequality for $Y$. His
result implies
%
%
\begin{equation}
\label{bous} \mpr \biggl(Y-(1+\epsilon)\me Y\geq\sigma\sqrt{2x}+ \biggl(
\frac
{1}{\epsilon}+\frac{1}{3} \biggr)x \biggr)\leq \mathrm{e}^{-nx}
\qquad\mbox{for all }x,\epsilon>0.
\end{equation}
For many statistical applications, it is important to have bounds
like $\sigma\sqrt{2x}+ (\frac{1}{\epsilon}+\frac{1}{3}
)x$ and
$\mathrm{e}^{-nx}$ that are, apart from the assumptions, completely
independent of the functions $f$; the parameter
$\epsilon> 0$ is inserted to obtain such bounds. Exponential inequalities
for bounded empirical processes similar to the one above have been
found by Klein and Rio
\citep{Klein05} and by Massart \citep{Massart00}. These inequalities
are slightly
less sharp, but additionally hold for nonidentically
distributed random variables and also for $-Y$. The derivations
of the mentioned results rely on the entropy method (initiated by
Ledoux in \citep{Ledoux97}), which provided a new approach to the
results in Talagrand's seminal work \citep{Talagrand96b}. For an
overview of the techniques involved, we refer to the textbooks
\cite{Boucheron13,LeTal91,Massart07}.

Results are also known for possibly unbounded empirical processes that have
weak tails. We consider independent and identically distributed random
variables $X_1,\ldots,X_n$ and a function class $\mathcal F$ such that
$\sup_{i,f\in\mathcal{F}}|\me f(X_i)|=0$ and
$\operatorname{card}\mathcal{F}=p$. We additionally assume that Bernstein
conditions are fulfilled, that is, $\sup_{f\in\mathcal{F}}\frac
{1}{n}\sum_{i=1}^n\me
|f(X_i)|^m\leq\frac{m!}{2}K^{m-2}$, $m=2,3,\ldots$ for a constant
$K$. B\"uhlmann and van de Geer then derive in \citep{Buhlmann11} the
following exponential deviation inequality for $Y:=\sup_{f\in\mathcal
{F}}|\frac{1}{n}\sum_{i=1}^nf(X_i)|$:
\[
\mpr \biggl(Y -\sqrt{\frac{2\log(2p)}{n}} -\frac{K\log(2p)}{n} \geq Kx +\sqrt{2x}
\biggr)\leq \mathrm{e}^{-nx}\qquad\mbox{for all }x>0.
\]
The lower bounds $Kx +\sqrt{2x}$ and
$\mathrm{e}^{-nx}$ are again independent of the functions $f$. Besides
the classical results for Gaussian processes (see, e.g.,
\cite{Boucheron13} and the references therein), other exponential
bounds for
unbounded
empirical processes are given by Adamczak in \citep{Adamczak08} and by van
de Geer and Lederer in \citep{vdGeer11b}. These authors assume weak
tails with respect to suitable Orlicz norms.

But what if the empirical process is unbounded and does not fulfill the
strict tail assumptions mentioned above? There is no hope to derive
exponential bounds as above under considerably weaker
assumptions. However, we show in the following that weak moment assumptions
are sufficient to obtain useful moment type
concentration inequalities. For this purpose, we
consider independent, not necessarily identically
distributed random variables $X_1,\ldots,X_n$ and a countable
function class $\mathcal F$ with an envelope that has $p$th moment at
most $M^p$ for a $M>0$ and a $p\in[1,\infty)$. Our main result,
Theorem~\ref{lemma.ConIn2.FirstCor}, implies then for
$Y:=\sup_{f\in\mathcal{F}}|\frac{1}{n}\sum_{i=1}^nf(X_i)|$,
$\sigma:=\sup_{f\in\mathcal{F}}\sqrt{\me[f(X_1)]^2}$, $1\leq
l\leq p$,
$(\cdot)_+:=\max\{0,\cdot\}$, $\|\cdot
\|_l:= (\me[ \cdot]^l )^{\trup{1}{l}}$ and for all
$\epsilon>0$
\[
\bigl\llVert \bigl(Y-(1+\epsilon)\me{Y} \bigr)_+\bigr\rrVert _l\leq
\biggl(\frac{64}{\epsilon}+7+\epsilon \biggr) \biggl(\frac
{l}{n}
\biggr)^{1-\trup{l}{p}} M+4\sqrt{\frac{l}{n}} \sigma
\]
and
\[
\bigl\llVert \bigl((1-\epsilon)\me{Y}-Y \bigr)_+\bigr\rrVert _l\leq
\biggl(\frac{86.4}{\epsilon}+7-\epsilon \biggr) \biggl(\frac
{l}{n}
\biggr)^{1-\trup{l}{p}} M+4.7\sqrt{\frac{l}{n}} \sigma.
\]
We argue in Section~\ref{sec.M1} that these bounds are
especially useful in
the common case where the envelope (measured by $M$) is much larger than
the single functions (measured by ${\sigma}$). We also stress that the
empirical process is present on the right-hand sides only through the
quantities $M$ and $\sigma$, which can be considered as properties of
the single random variables $f(X_i)$, unlike in known maximal
inequalities, which directly involve $\me Y$
or the entropy of the function set $\mathcal F$ (see, e.g.,
\citep{Massart07}, Chapter~6, and
\citep{vdVaart11}) at the corresponding spots. To obtain this, a
parameter $\epsilon>0$ is required as above.

We close this section with a short outline of the paper. In
Section~\ref{sec.Guideline}, we give the basic definitions and
assumptions. In
Section~\ref{sec.M1}, we then state and discuss the main result. This is
followed by complementary bounds in Section~\ref{sec.M2}. Detailed
proofs are finally given in Section~\ref{sec.proofs}.

\section{Random vectors, concentration inequalities and envelopes}
\label{sec.Guideline}

We are mainly interested in the behavior of \textit{suprema of empirical
processes}
%
%
\begin{equation}
\label{eq.intro.ep} Y:=\sup_{f\in\mathcal F}\Biggl\llvert \frac{1}{n}\sum
_{i=1}^nf(X_i)\Biggr\rrvert
\quad\mbox{or}\quad Y:=\sup_{f\in\mathcal F}\Biggl\llvert \frac
{1}{n}
\sum_{i=1}^n \bigl(f(X_i)-\me
f(X_i) \bigr)\Biggr\rrvert
\end{equation}
for large $n$. Here, $X_1,\ldots,X_n$ are independent, not necessarily
identically distributed random variables and $\mathcal{F}$ is a
countable family of
real, measurable functions. In the sequel, we may restrict
ourselves to finitely many functions by virtue of the monotonous convergence
theorem.

\textit{Random vectors} generalize the notion of empirical
processes. Let $\mathcal{Z}_1,\ldots,\mathcal{Z}_n$ be arbitrary probability
spaces and $\{Z_i(j)\dvtx\mathcal{Z}_i\to\mr, 1\leq j\leq N, 1\leq
i\leq
n\}$
a set of random variables. We then define the random vectors as
$Z(j):=(Z_1(j),\ldots,Z_n(j))^T\dvtx\mathcal{Z}_1\times
\cdots\times  \mathcal{Z}_n\to\mr^n$. For convenience, we introduce
their mean
as $ P Z(j):=\frac{1}{n}\sum_{i=1}^n\me Z_i(j)$, their empirical mean as
$ \mpr_nZ(j):=\frac{1}{n}\sum_{i=1}^nZ_i(j)$, and the root of their
maximal second moment as
$\sigma:=\max_{1\leq j \leq N}\sqrt{\frac{1}{n}\sum_{i=1}^n\me
Z_i(j)^2}$ (all
assumed to be finite). Throughout this paper, we
then consider the generalized formulation of \eqref{eq.intro.ep}
%
%
\begin{equation}
\label{eq.intro.rv} Z:=\max_{1\leq j\leq N}\bigl |\mprn Z(j)\bigr |.
\end{equation}
The corresponding results for the empirical
processes \eqref{eq.intro.ep} can be found via $Z_i(j):= f_j(X_i)$ or
$Z_i(j):=f_j(X_i)-\me
f_j(X_i)$ for $\mathcal F=\{f_1,\ldots,f_N\}$.

The basic assumption on the random vectors is expressed using
\textit{envelopes}. First, we call $\mathcal{E}:=(\mathcal
{E}_1,\ldots
,\mathcal{E}_n)^T \dvtx\mathcal{Z}_1\times
\cdots\times\mathcal{Z}_n\to\mr^n$ an envelope if $|Z_i(j)|\leq
\mathcal{E}_i$ for all $1\leq j\leq
N$ and $1\leq i \leq n$. The basic assumption of this paper is then that
there is a $p\in[1,\infty)$ and an $M>0$ such that
%
%
\begin{equation}
\me\mathcal{E}_i^p\leq M^p
\end{equation}
for all
$1\leq i \leq n$. To allow for an extension to countably infinite
families of functions $\mathcal F$ via the monotoneous convergence
theorem, we assume that the constant $M$ is independent of $N$.
Finally, we stress that the envelope $\mathcal{E}$ is typically much
larger than the single
random vectors $Z(j)$, that is, $M\gg\sigma$.

Asymptotically ($n\to\infty$), the processes \eqref{eq.intro.ep} and
\eqref{eq.intro.rv}
are typically governed by the central limit theorem. We study in this paper,
however, the nonasymtotic behavior ($n$ finite) of the process
\eqref{eq.intro.rv} (and
thus of \eqref{eq.intro.ep}). For $n$ finite, \textit{concentration
inequalities} provide bounds for the deviations in both directions from
the mean or
related quantities. Similarly, \textit{deviation
inequalities} provide bounds for the deviation in one
direction only. We are
especially interested in bounds that depend only on $n$, $M$, $\sigma
$, and
$p$. The bounds should, in particular, not depend on the functions
$f$ and therefore not on $\me Z$ or $\mathcal F$.

\section{Main result}
\label{sec.M1}

We are mainly concerned with concentration inequalities for unbounded
empirical processes that only
fulfill weak moment conditions. In particular, we are interested in bounds
that only depend on $n$, $M$, $\sigma$, and $p$ and incorporate empirical
processes with envelopes that may be much larger than the single
functions under consideration.

The following theorem is the main result of this paper.
%
%
\begin{theorem}
\label{lemma.ConIn2.FirstCor}
For $1\leq l\leq p$ and all $\epsilon>0$ it holds that
\[
\bigl\llVert \bigl(Z-(1+\epsilon)\me{Z} \bigr)_+\bigr\rrVert _l\leq
\biggl(\frac{64}{\epsilon}+7+\epsilon \biggr) \biggl(\frac
{l}{n}
\biggr)^{1-\trup{l}{p}} M+4\sqrt{\frac{l}{n}} \sigma
\]
and
\[
\bigl\llVert \bigl((1-\epsilon)\me{Z}-Z \bigr)_+\bigr\rrVert _l\leq
\biggl(\frac{86.4}{\epsilon}+7-\epsilon \biggr) \biggl(\frac
{l}{n}
\biggr)^{1-\trup{l}{p}} M+4.7\sqrt{\frac{l}{n}} \sigma.
\]
\end{theorem}

As discussed in the preceding section, we state our results in
terms of random vectors instead of empirical processes. The connection can
be made as described. Furthermore, we note that a considerable improvement
with respect to $l$ does not seem to be possible. Slightly better constants
can be obtained, however, at the price of less incisive bounds or less
accessible proofs (see Remark~\ref{rm.constants} in the proofs
section). We finally
note that the expectation $\me Z$ can be replaced by suitable
approximations. Such approximations are usually found with chaining and
entropy (see, e.g., \citep{dudley1967sizes,vdVaart00,vdVaart11})
or generic chaining (see, e.g.,
\citep{Fernique75,Talagrand96,Talagrand05}).

Let us now have a closer look at the above result. In contrast to the
known results given in the introduction, the single functions may be
unbounded and may only fulfill weak moment conditions. For the
envelope, the moment restrictions are increasing with
increasing power $l$, as expected.

And what about large envelopes, that is $M\gg
{\sigma}$? Theorem~\ref{lemma.ConIn2.FirstCor}
separates the part including the size of the envelope (measured by $M$)
from the part
including the size of the single random vectors (measured by
${\sigma}$). For $p>2l$ and $n
\gg1$, a possibly large value of $M$ is counterbalanced by
$\frac{1}{n^{1-\trup{l}{p}}}\ll\frac{1}{\sqrt n} $ and thus, the influence
of large envelopes is tempered. In particular, the term
including the size of the envelope can be neglected for
$n\to\infty$ if $p$ is sufficiently large.

We conclude this section with two straightforward consequences of
Theorem~\ref{lemma.ConIn2.FirstCor} and an additional remark.
%
%
\begin{corollary}
Theorem~\ref{lemma.ConIn2.FirstCor} directly implies
probability bounds via
Chebyshev's inequality. Under the above assumptions, it holds for $x > 0$
\[
\mpr \bigl(Z\geq(1+\epsilon)\me{Z} +x \bigr)\leq\min_{1\leq l
\leq
p}
\frac{ (  ((\trup{64}{\epsilon})+7+\epsilon )
(\trup{l}{n} )^{1-\trup{l}{p}} M+4\sqrt\trup{l}{n} \sigma
)^l}{x^l}
\]
and similarly
\[
\mpr \bigl(Z\leq(1-\epsilon)\me{Z} -x \bigr)\leq\min_{1\leq l
\leq
p}
\frac{ (  ((\trup{86.4}{\epsilon})+7-\epsilon
)
(\trup{l}{n} )^{1-\trup{l}{p}} M+4.7\sqrt{\trup{l}{n}}
\sigma )^l}{x^l}
.
\]
\end{corollary}

If $\sigma$, $M$, $\epsilon$, and $ p-2l$ are strictly positive
constants, this implies the logarithmic rate $(\ln(n))^{- l/2}$
for $\mpr (Y - (1+\epsilon)\me{Y} \geq\sqrt{\ln(n)/n}
)$. This rate can be directly compared to the corresponding polynomial
rates resulting from \eqref{bous} (with $x\sim\ln(n)/n$) to observe
that the avoiding of the boundedness assumption causes slower rates, as
expected. 

%
\begin{corollary}
Concrete first order bounds under the above
assumptions are for example
\[
\me [Z-2\me{Z} ]_+ \leq 72 \frac{M}{n^{1-\trup{1}{p}}} + 4 \frac{\sigma}{\sqrt n}
\]
and
\[
\me \biggl[\frac{1}{2}\me{Z}-Z \biggr]_+ \leq179.3 \frac{M}{n^{1-\trup{1}{p}}} +
4.7 \frac{\sigma}{\sqrt n}.
\]
\end{corollary}

%
\begin{remark}
Allowing the right-hand side in Theorem~\ref{lemma.ConIn2.FirstCor} to
depend on $\me Z$, we can avoid
the parameter $\epsilon$ and find for example
\[
\bigl\llVert (Z-\me Z )_+\bigr\rrVert _l\leq 10.2 \biggl(
\frac{l}{n} \biggr)^{1-\trup{l}{p}} M+\sqrt{32 \biggl(\frac{l}{n}
\biggr)^{1-\trup{l}{p}}M\me Z}+ \sqrt{\frac{2l}{n}} \sigma
\]
and a similar bound for $\llVert (\me Z-Z)_+\rrVert _l$. For the
proof, one
can proceed similarly as in
Remark~\ref{rm.constants} and Lemma~\ref{thm.ConIn2.wmass2}. These
kinds of results are, however, of less
statistical importance.
\end{remark}

\section{Complementary bounds}
\label{sec.M2}

In this section, we complement the main result
Theorem~\ref{lemma.ConIn2.FirstCor} with two additional bounds.
These
additional bounds can be of interest if $l$ is close to $p$.

The first result reads as the following theorem.
%
%
\begin{theorem}\label{theorem.Massart2} Assume that the random variables
$Z_i(j)$ are centered. For $1\leq l\leq p$ it holds that
\[
\bigl\llVert (Z-4\me Z )_+\bigr\rrVert _l\leq \bigl(l \Gamma (l/2 )
\bigr)^{\trup{1}{l}}\sqrt{\frac
{32}{n}} M,
\]
where ${\Gamma}$ is the usual Gamma function.
\end{theorem}

Let us compare Theorem~\ref{theorem.Massart2}
with Theorem~\ref{lemma.ConIn2.FirstCor}. On the one hand, the above
result does not possess the flexibility of the factor $(1+\epsilon)$
and is
a deviation inequality only. On the other hand, the term including the size
of the envelope $M$ is independent of $p$ and has a
different power of $n$ in the denominator compared to
the corresponding term in Theorem~\ref{lemma.ConIn2.FirstCor}. Comparing
these two terms in detail, we find that the bound of
Theorem~\ref{lemma.ConIn2.FirstCor} may be sharper than the
corresponding bound in Theorem~\ref{theorem.Massart2} if $l\leq p <
2l$.

We finally give explicit deviation inequalities for $Z$ in the case of
finitely many random
vectors. For $p\geq2$, explicit bounds are found immediately by
replacing $\me Z$ in
Theorem~\ref{lemma.ConIn2.FirstCor} or Theorem~\ref{theorem.Massart2}
by the upper bound $\sqrt{\frac{8\log
(2N)}{n}}M$ (see \cite{Duembgen09}). Another bound is found by an
approach detailed in Section~\ref{sec.proofs}. The bound reads as follows.
%
%
\begin{theorem}\label{sec4.Corrollary1}
Assume that the random variables
$Z_i(j)$ are centered. Then, for
$p\geq2$, $l\in\mn$, and $p\geq l$,
\[
\biggl\llVert \biggl(Z-2M\frac{\log(2N)}{\sqrt n} \biggr)_+\biggr\rrVert
_l\leq \sqrt{\frac{35}{n}} l M.
\]
\end{theorem}

This can supersede the bound in Theorem~\ref{theorem.Massart2}
for $\log(2N)\leq32$.

\section{Proofs}
\label{sec.proofs}

In this last section, we give detailed proofs.

\subsection{Proof of Theorem \texorpdfstring{\protect\ref{lemma.ConIn2.FirstCor}}{3.1}}

The key idea of our proofs is to introduce an appropriate truncation that
depends on the envelope of the
empirical process. This allows us to split the problem into two parts
that can be
treated separately: On the one hand, a part corresponding to a bounded
empirical process that can be treated by convexity
arguments and Massart's
results on bounded random vectors \citep{Massart00}. And
on the other hand, a part corresponding to an unbounded empirical process
that can be treated by rather elementary means.

For ease of exposition, we present some convenient
notation for the truncation first. After deriving a simple auxiliary
result, we then turn
to the main task of this section: We first consider the truncated part of
the problem in Lemma~\ref{lemma.boundstrunc}
and then prove Lemma~\ref{thm.ConIn2.wmass2}, a generalization
of Theorem~\ref{lemma.ConIn2.FirstCor}.

A basic tool used in this section is \textit{truncation}. Before
turning to the proofs, we want to give some additional notation for this
tool. First,\vadjust{\goodbreak} we define the \textit{unbounded} and the
\textit{bounded part of the random vectors} as
\begin{eqnarray*}
\overline{Z}(j)&:=&\bigl(\overline{Z}_1(j),\ldots,\overline
{Z}_n(j)\bigr)^T:=\bigl(Z_1(j)1_{\{\mathcal{E}_1>K\}},
\ldots,Z_n(j)1_{\{
\mathcal{E}_n>K\}}\bigr)^T,
\\
\underline{Z}(j)&:=&\bigl(\underline{Z}_1(j),\ldots,\underline
{Z}_n(j)\bigr)^T:=\bigl(Z_1(j)1_{\{\mathcal{E}_1\leq
K\}},
\ldots,Z_n(j)1_{\{\mathcal{E}_n\leq
K\}}\bigr)^T.
\end{eqnarray*}
Similarly, we define
\begin{eqnarray*}
\overline{\mathcal E}&:=&(\overline{\mathcal{E}}_1,\ldots,\overline
{\mathcal{E}}_n)^T:=(\mathcal{E}_11_{\{\mathcal{E}_1>K\}},
\ldots ,\mathcal{E}_n1_{\{\mathcal{E}_n>K\}})^T,
\\
\underline{\mathcal{E}}&:=&(\underline{\mathcal{E}}_1,\ldots ,
\underline{\mathcal{E}}_n)^T:=(\mathcal{E}_11_{\{\mathcal{E}_1\leq
K\}},
\ldots,\mathcal{E}_n1_{\{\mathcal{E}_n\leq
K\}})^T.
\end{eqnarray*}
To prevent an overflow of indices, the \textit{truncation level}
$K>0$ is not included explicitly in the notation. The
truncation level is, however, given at the adequate places so that there
should not be any confusion. Finally, we define the maxima of the
truncated random variables as
\[
\ovez:=\max_{1\leq j\leq
N}\bigl |\mprn\ovez(j)\bigr |\quad\mbox{and}\quad \undz:=
\max_{1\leq j\leq
N}\bigl |\mprn\undz(j)\bigr |
\]
and the maximal variance of the bounded parts as
\[
{\underline{\sigma}}:=\max_{1\leq j \leq N}\sqrt{
\frac{1}{n}\sum_{i=1}^n
\operatorname{Var}\undz_i(j)}.
\]
Now we derive a simple auxiliary lemma.
%
%
%
%
\begin{lemma}\label{lemma.ConIn.Verg}
Under the assumptions of Theorem~\ref{lemma.ConIn2.FirstCor}, it holds
that $\sigma\geq\underline{\sigma}$ and
\[
\bigl |\me[\undz-Z]\bigr |^l\leq\frac{M^p}{K^{p-l}}
\]
for the truncation level $K>0$.
\end{lemma}

\begin{pf}
The first assertion is straightforward. For the second assertion, since
$||a|-|b||\leq|a-b|$ for all $a,b\in\mr$, we observe that
\begin{eqnarray*}
\bigl\llvert \me[\undz-Z]\bigr\rrvert &=&\Bigl\llvert \me \Bigl[\max
_{1\leq j\leq
N}\bigl |\mprn\undz(j)\bigr |-\max_{1\leq j\leq
N}\bigl |\mprn Z(j)\bigr |
\Bigr]\Bigr\rrvert
\\
&\leq&\me \Bigl[\max_{1\leq j\leq
N}\bigl |\bigl |\mprn\undz(j)\bigr |-\bigl |\mprn Z(j)\bigr |\bigr |
\Bigr]
\\
&\leq&\me \Bigl[\max_{1\leq j\leq
N}\bigl |\mprn\bigl( \undz(j)-Z(j)\bigr)\bigr |
\Bigr]
\\
&=&\me \Bigl[\max_{1\leq j\leq
N}|\mprn\ovez| \Bigr]
\\
&\leq&\me \Biggl[\frac{1}{n}\sum_{i=1}^n
\overline{\mathcal {E}}_i \Biggr].
\end{eqnarray*}
With H\"older's and Chebyshev's inequality, we obtain for $1\leq i\leq n$
\begin{eqnarray*}
\me\overline{\mathcal{E}}_i^l &=& \me
\mathcal{E}_i^l 1_{\{\mathcal{E}_i>K\}}
\\
&\leq& \bigl(\me\mathcal{E}_i^p\bigr)^{\trup{l}{p}} (
\me 1_{\{\mathcal{E}_i>K\}})^{1-\trup{l}{p}}
\\
&\leq& \bigl(\me\mathcal{E}_i^p\bigr)^{\trup{l}{p}}
\biggl(\frac{\me
\mathcal
{E}_i^p}{K^p} \biggr)^{1-\trup{l}{p}}
\\
&\leq& \frac{M^p}{K^{p-l}}.
\end{eqnarray*}
These two results and Jensen's inequality yield then the second assertion.\vadjust{\goodbreak}
\end{pf}

We can now turn to the harder part of this section. We first consider
bounded random vectors in Lemma~\ref{lemma.boundstrunc}. We then proof
a bound for
unbounded random vectors in Lemma~\ref{thm.ConIn2.wmass2}, from which
the main result can be
deduced easily.

%
\begin{lemma}
\label{lemma.boundstrunc}
Let $1\leq l\leq p$, $\epsilon>0$ and denote by $K>0$ the truncation
level. Then,
\[
\bigl\llVert \bigl(\underline{Z}-(1+\epsilon)\me\underline{Z} \bigr)_+\bigr
\rrVert _l\leq \biggl(\frac{64}{\epsilon}+5 \biggr)\frac
{lK}{n}+
\frac{4\sqrt{l}\underline{\sigma}}{\sqrt n}
\]
and
\[
\bigl\llVert \bigl((1-\epsilon)\me\underline{Z}-\underline{Z} \bigr)_+\bigr
\rrVert _l\leq \biggl(\frac{86.4}{\epsilon}+5 \biggr)\frac{lK}{n}+
\frac{4.7\sqrt {l}\underline{\sigma}}{\sqrt n}.
\]
\end{lemma}

\begin{pf}
The key idea is to use convexity arguments so that we can apply well-known
bounds for bounded random vectors.

To begin,
we set $ J:= ({32}/{\epsilon}+2.5 )K$ and $I:= ({2(l
-1)J}/n+{\sqrt{8(l
-1)}\underline{\sigma}}/{\sqrt n} )^l$
and then define the
function $g_l\dvtx\mr^+\to(1,\infty)$ as
\[
g_l (x):=\mathrm{e}^{(\trup{n}{(2J^2)}) (\sqrt{2\underline
{\sigma
}^2+J(x\vee I)^{\trup{1}{l}}}-\sqrt{2}\underline{\sigma} )^2}.
\]
We used here the notation $a\vee
b:=\max\{a,b\}$ for $a,b\in\mr$. The function $g_l$ is strictly
increasing, smooth, convex on the interval $(I,\infty)$, and its
inverse on $(1,\infty)$ is given by
%
%
\begin{equation}
\label{eq.inverse} g_l^{-1}(y)= \biggl(\frac{2J}{n}
\log{y}+\frac{4\underline
{\sigma}}{\sqrt n}\sqrt{\log y} \biggr)^l.
\end{equation}
The straightforward derivations of these facts are omitted for the sake
of brevity.

The convexity of the function $g_l$ makes it possible to apply a result
of \citep{Massart00}. To show this, we introduce
\[
X:=\bigl(\undz-(1+\epsilon)\me\undz\bigr)_+^l
\]
and find with Jensen's inequality and the fact that $g_l$ is increasing
\[
g_l(\me X)\leq g_l\bigl(\me[X\vee I]\bigr)\leq\me
g_l(X\vee I).
\]
Massart's inequality \citep{Massart00}, Theorem~4,
(13), for bounded random vectors translates then to our
setting as
\[
\mpr\biggl(n\undz\geq(1+\epsilon)n\me\undz+\underline{\sigma}\sqrt {8nx}+
\biggl(\frac{32}{\epsilon}+2.5 \biggr)Kx\biggr)\leq\mathrm{e}^{-x},
\]
where $x >0$. This is equivalent to
%
%
\begin{equation}
\label{eq.ConIn2.massart} \mpr\biggl(\undz\geq(1+\epsilon)\me\undz+\underline{\sigma}\sqrt
{\frac{8x}{n}}+\frac{J}{n}x\biggr)\leq\mathrm{e}^{-x}.
\end{equation}
We now deduce (cf. \cite{vdGeer11b})
\begin{eqnarray*}
&&\me \bigl[ \mathrm{e}^{(\trup{n}{(2J^2)}) (\sqrt
{2\underline{\sigma
}^2+J(X\vee I)^{\trup{1}{l}}}-\sqrt{2}\underline{\sigma}
)^2} \bigr]
\\
&&\quad= \int_0^\infty\mpr \bigl(
\mathrm{e}^{(\trup
{n}{(2J^2)}) (\sqrt
{2\underline{\sigma}^2+J(X\vee I)^{\trup{1}{l}}}-\sqrt
{2}\underline
{\sigma} )^2}> t \bigr)\,\mathrm{d}t
\\
&&\quad\leq1+\int_1^\infty\mpr \bigl(
\mathrm{e}^{(\trup
{n}{(2J^2)}) (\sqrt
{2\underline{\sigma}^2+J(X\vee I)^{\trup{1}{l}}}-\sqrt
{2}\underline
{\sigma} )^2}> t \bigr)\,\mathrm{d}t
\\
&&\quad= 1+\int_1^\infty\mpr \biggl(\sqrt{2
\underline{\sigma }^2+J(X\vee I)^{\trup{1}{l}}}> \sqrt{2}\underline{
\sigma}+\sqrt{\frac
{2J^2}{n}\log t} \biggr)\,\mathrm{d}t
\\
&&\quad= 1+\int_1^\infty\mpr \biggl(J(X\vee
I)^{\trup{1}{l}}> 4\underline {\sigma}\sqrt{\frac{J^2}{n}\log t}+
\frac{2J^2}{n}\log t \biggr)\,\mathrm{d}t
\end{eqnarray*}
and note that
\begin{eqnarray*}
JI^{\trup{1}{l}} &<& 4\underline{\sigma}\sqrt{\frac{J^2}{n}\log t}+
\frac{2J^2}{n}\log t\quad\Leftrightarrow
\\
\frac{2(l
-1)J}{n}+\frac{\sqrt{8(l
-1)}\underline{\sigma}}{\sqrt n} &<&4\underline{\sigma}\sqrt{
\frac
{\log t}{n}}+ \frac{2J}{n}\log t.
\end{eqnarray*}
This is fulfilled if $t\geq \mathrm{e}^{l-1}$. Hence, with Massart's inequality
\eqref{eq.ConIn2.massart},
\begin{eqnarray*}
&&\me \bigl[ \mathrm{e}^{(\trup{n}{(2J^2)}) (\sqrt
{2\underline{\sigma
}^2+J(X\vee I)^{\trup{1}{l}}}-\sqrt{2}\underline{\sigma}
)^2} \bigr]
\\
&&\quad\leq1+\mathrm{e}^{l-1}-1+\int_{\mathrm{e}^{l-1}}^\infty
\mpr \biggl(X^{\trup{1}{l}}> 4\underline{\sigma}\sqrt{\frac{\log t}{n}}+
\frac{2J}{n}\log t \biggr)\,\mathrm{d}t
\\
&&\quad= \mathrm{e}^{l-1}+\int_{\mathrm{e}^{l-1}}^\infty
\mpr \biggl(\undz >(1+\epsilon)\me \undz+4\underline{\sigma}\sqrt{
\frac{\log t}{n}}+ \frac
{2J}{n}\log t \biggr)\,\mathrm{d}t
\\
&&\quad\leq \mathrm{e}^{l-1}+\int_{\mathrm{e}^{l-1}}^\infty
\exp\bigl(-\log t^2\bigr)\,\mathrm{d}t < \mathrm{e}^{l}.
\end{eqnarray*}
In summary, we have
\[
g_l(\me X)< \mathrm{e}^l.
\]
This is now inverted using equation \eqref{eq.inverse} to obtain
%
%
%
\[
\me X\leq \biggl(\frac{2lJ}{n}+\frac{4\sqrt{l}\underline{\sigma
}}{\sqrt n} \biggr)^l.
\]
This finishes the proof of the first claim. The second claim can be
deduced similarly using \citep{Massart00}, Theorem~4, (14).
\end{pf}

%
\begin{remark}
\label{rm.constants}
The constants in Lemma~\ref{lemma.boundstrunc} are not optimal. First, we
note that more restrictive assumptions allow one to replace Massart's
inequality \eqref{eq.ConIn2.massart} by sharper concentration inequalities
(e.g., from Klein and Rio \citep{Klein05} assuming
centered random vectors or from Bousquet \citep{Bousquet02} assuming
centered and identically
distributed random vectors) and permit therefore shaper bounds. Second,
instead of using
such concentration inequalities, one can work with the underlying log-Laplace
transforms directly. We found that this approach leads to slightly better
constants but also to a less accessible proof. Let us sketch the
approach:

One may first verify that for any $t>0$ and $a\geq
0$
%
%
\begin{equation}
\label{eq.firstloglap} \log\me \bigl[n(\undz-\me\undz)-a \bigr]_+^l\leq\log
\me \bigl[\mathrm{e}^{tn(\undz-\me\undz)}\bigr]+l\log(l/t)-l-ta.
\end{equation}
We can now use bounds for the log-Laplace transform $\log\me
[\mathrm{e}^{tn(\undz-\me\undz)}]$ of $n(\undz-\me\undz)$, for
example, from
\citep{Massart00}:
%
%
\begin{equation}
\label{eq.loglap} \log\me\bigl[\mathrm{e}^{tn(\undz-\me\undz)}\bigr]\leq
\frac
{vt^2}{1-2.5Kt}\qquad\mbox{with }v:={2n\underline{\sigma}^2+32Kn\me
\undz}.
\end{equation}
The bound \eqref{eq.loglap} with
$t:= (\sqrt{\frac{v}{l}}+2.5K )^{-1}$ can be inserted
into \eqref{eq.firstloglap} and the result can be simplified with
\[
( \sqrt{\alpha+\beta}+\gamma)\mathrm{e}^{
-\trup{\alpha/\delta}{(\sqrt{\alpha+\beta}+\gamma)}}\leq\sqrt {\beta}+2\gamma+
\delta
\]
for $\alpha,\beta,\gamma,\delta>0$. This then leads to the bound
\[
\bigl\llVert \bigl(\underline{Z}-(1+\epsilon)\me\underline{Z} \bigr)_+\bigr
\rrVert _l\leq \biggl(\frac{32}{\epsilon}+5 \biggr)\frac
{lK}{n}+
\frac{\sqrt{2l}\underline{\sigma}}{\sqrt n}.
\]
The quantity $\llVert  ((1-\epsilon)\me\underline{Z}-\underline
{Z} )_+\rrVert _l$
can be bounded similarly.
\end{remark}

We now use the above lemma to prove the following.
%
%
\begin{lemma}
\label{thm.ConIn2.wmass2}
Let $1\leq l\leq p$, $\epsilon>0$ and denote by $K>0$ the truncation
level. Then,
\[
\bigl\llVert \bigl({Z}-(1+\epsilon)\me\underline{Z} \bigr)_+\bigr\rrVert
_l\leq \biggl(\frac{64}{\epsilon}+5 \biggr)\frac{lK}{n}+
\frac{4\sqrt
l\underline{\sigma}}{\sqrt
n}+\frac{M^{\trup{p}{l}}}{K^{\trup{p}{l}-1}}
\]
and
\[
\bigl\llVert \bigl((1-\epsilon)\me\underline{Z}-{Z} \bigr)_+\bigr\rrVert
_l \leq \biggl(\frac{86.4}{\epsilon}+5 \biggr)\frac{lK}{n}+
\frac{4.7\sqrt
l\underline{\sigma}}{\sqrt
n}+\frac{M^{\trup{p}{l}}}{K^{\trup{p}{l}-1}}.
\]
\end{lemma}

\begin{pf}
The key idea of the proof is to separate the bounded from the unbounded
quantities. We then develop bounds for $\llVert \overline{Z}\rrVert
_l$ via elementary means
and combine this with the above results to deduce the desired bounds.

We start with the proof of the first inequality. First, we split $Z$ in
a bounded and an unbounded part
\begin{eqnarray*}
Z &=&\max_{1\leq j \leq N}\bigl |\mprn Z(j)\bigr |
\\
&=&\max_{1\leq j \leq N}\bigl |\mprn\bigl( \undz(j)+\ovez(j)\bigr)\bigr |
\\
&\leq&\max_{1\leq j \leq N}\bigl(\bigl |\mprn\undz(j)\bigr |+\bigl |\mprn\ovez(j)\bigr |\bigr)
\\
&\leq&\underline{Z}+\overline{Z}
\end{eqnarray*}
and deduce with the triangle inequality that
%
%
\begin{eqnarray}
\label{eq.ConIn2.splitting2} && \bigl\llVert \bigl(Z-(1+\epsilon)\me\underline{Z} \bigr)_+\bigr
\rrVert _l
\nonumber
\\
&&\quad\leq\bigl\llVert \bigl(\underline{Z}+\overline{Z}-(1+\epsilon )\me
\underline{Z} \bigr)_+\bigr\rrVert _l
\nonumber
\\[-8pt]
\\[-8pt]
&&\quad\leq\bigl\llVert \bigl(\underline{Z}-(1+\epsilon)\me\underline {Z}
\bigr)_++\overline{Z}\bigr\rrVert _l
\nonumber
\\
&&\quad\leq\bigl\llVert \bigl(\underline{Z}-(1+\epsilon)\me \underline{Z}
\bigr)_+\bigr\rrVert _l+\llVert \overline{Z}\rrVert _l.
\nonumber
\end{eqnarray}
Now, we turn to the development of bounds for
$\llVert \overline{Z}\rrVert _l$. As above, with the help of
H\"older's and Chebyshev's inequalities, we
obtain for $1\leq i\leq n$
\[
\me\overline{\mathcal{E}}_i^l \leq\frac{M^p}{K^{p-l}}
\]
and therefore with the triangle inequality
%
%
\begin{equation}
\label{eq.ConIn2.Chebby} \llVert \overline{Z}\rrVert _l\leq\llVert \mprn
\mathcal{\overline {E}}\rrVert _l\leq\frac{M^{\trup{p}{l}}}{K^{\trup{p}{l}-1}}.
\end{equation}
Combining inequalities \eqref{eq.ConIn2.splitting2},
\eqref{eq.ConIn2.Chebby}, and the bound from Lemma~\ref{lemma.boundstrunc}
gives finally
\[
\bigl\llVert \bigl(Z-(1+\epsilon)\me\underline{Z} \bigr)_+\bigr\rrVert
_l \leq \biggl(\frac{64}{\epsilon}+5 \biggr)\frac{lK}{n}+
\frac{4\sqrt {l}\underline{\sigma}}{\sqrt
n} + \frac{M^{\trup{p}{l}}}{K^{\trup{p}{l}-1}}.
\]
This finishes the proof of the first part of the lemma. For the second
part, we note that
\begin{eqnarray*}
\undz&=&\max_{1\leq j \leq N}\bigl |\mprn\undz(j)\bigr |
\\
&=&\max_{1\leq j \leq N}\bigl |\mprn\bigl(Z(j)-\ovez(j)\bigr)\bigr |
\\
&\leq&\max_{1\leq j \leq N}\bigl(\bigl |\mprn Z(j)\bigr |+\bigl |\mprn\ovez(j)\bigr |\bigr)
\\
&\leq& Z+\overline{Z}
\end{eqnarray*}
and therefore $Z\geq\undz- \ovez$. Consequently,
\begin{eqnarray*}
&&\bigl\llVert \bigl((1-\epsilon)\me\underline{Z}-Z \bigr)_+\bigr\rrVert
_l
\\
&&\quad\leq\bigl\llVert \bigl((1-\epsilon)\me\underline{Z}-\undz +\ovez \bigr)_+
\bigr\rrVert _l
\\
&&\quad\leq\bigl\llVert \bigl((1-\epsilon)\me\underline{Z}-\undz \bigr)_++
\overline{Z}\bigr\rrVert _l
\\
&&\quad\leq\bigl\llVert \bigl((1-\epsilon)\me \underline{Z}-\undz \bigr)_+\bigr
\rrVert _l+\llVert \overline{Z}\rrVert _l.
\end{eqnarray*}
One can then proceed as in the first part.
\end{pf}

\begin{pf*}{Proof of Theorem~\ref{lemma.ConIn2.FirstCor}}
Set $K= (\frac{n}{l} )^{\trup{l}{p}}M$ in Lemma~\ref
{thm.ConIn2.wmass2} and use
Lemma~\ref{lemma.ConIn.Verg} to replace the truncated quantities by the
original ones.
\end{pf*}

%
%

\subsection{Proof of Theorem \texorpdfstring{\protect\ref{theorem.Massart2}}{4.1}}

Here, we prove Theorem~\ref{theorem.Massart2} with the help of
symmetrization and desymmetrization.

\begin{pf*}{Proof of Theorem~\ref{theorem.Massart2}}
The trick is to use symmetrization and desymmetrization arguments so
that we are
able to use \citep{Massart00}, Theorem~9, in a favorable way.

Beforehand, we define $Z_\epsilon:=\max_{1\leq j\leq N}|\frac
{1}{n}\sum_{i=1}^n
\epsilon_i Z_i(j)|$ with independent Rademacher random variables
$\epsilon_i$. Then, we symmetrize according to \citep{vdVaart00}, Lemma~2.3.6, with the function $\Phi(x)=(x-4\me Z)_+^l$ to obtain
\[
\me [Z-4\me Z ]_+^l\leq\me [2Z_\epsilon-4\me Z
]_+^l
\]
and we desymmetrize with the function $\Phi(x)=x$ to obtain
\[
\me [2Z_\epsilon-4\me Z ]_+^l\leq\me [2Z_\epsilon -
\me2Z_\epsilon ]_+^l.
\]
Hence,
%
%
\begin{equation}
\label{eq.Massartzwo1} \me [Z-4\me Z ]_+^l\leq2^l\me
\me_\epsilon [Z_\epsilon-\me Z_\epsilon ]_+^l,
\end{equation}
where we write here and in the following $\me_\epsilon$ for the
expectation and $\mpr_\epsilon$ for the probability w.r.t. the
Rademacher random variables. Next, we observe that
\begin{eqnarray*}
&&\me_\epsilon [Z_\epsilon-\me Z_\epsilon ]_+^l
\\
&&\quad= \int_0^\infty\mpr_\epsilon \bigl(
(Z_\epsilon-\me Z_\epsilon )_+^l>t \bigr)\,\mathrm{d}t
\\
&&\quad= \int_0^\infty\mpr_\epsilon
\bigl(Z_\epsilon>\me Z_\epsilon+ t^{\trup{1}{l}} \bigr)\,\mathrm{d}t
\\
&&\quad\leq\int_0^\infty\mpr_\epsilon
\Biggl(\max_{1\leq j\leq
N}\frac{1}{n}\sum
_{i=1}^n \epsilon_iZ_i(j)>
\me\max_{1\leq j\leq N}\frac{1}{n}\sum_{i=1}^n
\epsilon_iZ_i(j) + t^{\trup{1}{l}} \Biggr)\,\mathrm{d}t
\\
&&\qquad{}+\int_0^\infty\mpr_\epsilon
\Biggl(\max_{1\leq j\leq N}-\frac{1}{n}\sum
_{i=1}^n\epsilon_iZ_i(j)>
\me\max_{1\leq j\leq
N}\frac{1}{n}\sum_{i=1}^n
\epsilon_iZ_i(j) + t^{\trup{1}{l}} \Biggr)\,\mathrm{d}t
\\
&&\quad\leq2\int_0^\infty\mpr_\epsilon
\Biggl(\max_{1\leq j\leq
N}\frac{1}{n}\sum
_{i=1}^n \epsilon_iZ_i(j)>
\me\max_{1\leq j\leq N}\frac{1}{n}\sum_{i=1}^n
\epsilon_iZ_i(j) + t^{\trup{1}{l}} \Biggr)\,
\mathrm{d}t.
\end{eqnarray*}
In a final step, we apply Massart's inequality \citep{Massart00}, Theorem~9, with
\[
L^2=\max_{1\leq j\leq N}\sum_{i=1}^n
\bigl(2\bigl |Z_i(j)\bigr | \bigr)^2\leq 4n\mprn
\mathcal{E}^2,
\]
where $\mprn\mathcal{E}^2:=\frac{1}{n}\sum_{i=1}^n\mathcal
{E}_i^2$. This yields
\begin{eqnarray*}
&& 2\int_0^\infty\mpr_\epsilon \Biggl(\max
_{1\leq j\leq N}\frac
{1}{n}\sum_{i=1}^n
\epsilon_iZ_i(j)>\me\max_{1\leq j\leq N}
\frac{1}{n}\sum_{i=1}^n
\epsilon_iZ_i(j) + t^{\trup{1}{l}} \Biggr)\,\mathrm{d}t
\\
&&\quad\leq 2\int_0^\infty\exp \biggl( -
\frac{nt^{\trup{2}{l}}}{8\mprn
\mathcal{E}^2} \biggr)\,\mathrm{d}t
\\
&&\quad= 2 \biggl(\frac{8}{n} \biggr)^{\trup{l}{2}} \bigl(\mprn
\mathcal{E}^2\bigr)^{\trup{l}{2}}\int_0^\infty
\exp \bigl( -t^{\trup{2}{l}} \bigr)\,\mathrm{d}t
\\
&&\quad= 2 \biggl(\frac{8}{n} \biggr)^{\trup{l}{2}}\bigl(\mprn\mathcal
{E}^2\bigr)^{\trup
{l}{2}}\frac{l\Gamma (\trup{l}{2} )}{2}.
\end{eqnarray*}
With inequality \eqref{eq.Massartzwo1}, this gives
\[
\me [Z-4\me Z ]_+^l\leq2^ll \biggl(\frac{8}{n}
\biggr)^{\trup{l}{2}}\me \bigl[\mprn\mathcal{E}^2
\bigr]^{\trup
{l}{2}}\Gamma \biggl(\frac{l}{2} \biggr).
\]
Finally, due to the triangle inequality, it holds that
\[
\me \bigl[\mprn\mathcal{E}^2 \bigr]^{\trup{l}{2}}\leq\me [\mprn
\mathcal{E} ]^{l}\leq M^l
\]
and hence
\[
\me [Z-4\me Z ]_+^l\leq l\Gamma \biggl(\frac{l}{2} \biggr)
\biggl(\frac{32}{n} \biggr)^{\trup{l}{2}}M^l.
\]
\upqed
\end{pf*}

\subsection{Proof of Theorem \texorpdfstring{\protect\ref{sec4.Corrollary1}}{4.2}}

We eventually derive Theorem~\ref{sec4.Corrollary1} using
truncation. After some auxiliary results, we derive
Lemma~\ref{lemma.helpwoe}. This lemma settles the bounded part of the
problem. It is then used to proof Lemma~\ref{theorem.mainwoe} which is
a slight
generalization of the main theorem. Finally, we derive
Theorem~\ref{sec4.Corrollary1} as a simple corollary.

We begin with two auxiliary lemmas.\vadjust{\goodbreak}
%
%
\begin{lemma}
\label{lemma.helpfirst}
Let $W$ be a centered random variable with values in $[-A,A]$, $A\geq
0$, such that $\me W^2\leq1$. Then,
\[
\me\mathrm{e}^{\trup{W}{A}}\leq1 + \frac{1}{A^2}.
\]
\end{lemma}

\begin{pf}
We follow well known ideas (see, e.g., \cite{Buhlmann11}, Chapter~14):
\begin{eqnarray*}
\me\mathrm{e}^{\trup{W}{A}} &=& 1 + \me \biggl[ \mathrm{e}^{\trup
{W}{A}}-1-
\frac
{W}{A} \biggr]
\\
&\leq&1 + \me \biggl[ \mathrm{e}^{\trup{|W|}{A}}-1-\frac
{|W|}{A} \biggr]
\\
&=&1+\sum_{m=2}^\infty\frac{\me|W|^m}{m!A^m}
\\
&\leq&1+\sum_{m=2}^\infty\frac{A^{m-2}}{m!A^m}
\\
&\leq&1 +\frac{1}{A^2}.
\end{eqnarray*}
\upqed
\end{pf}

%
\begin{lemma}
\label{lemma.ConIn4.Combi}
Let $C_m^n:=|\{ (i_1,\ldots,i_m)^T\in\{ 1,\ldots,n \}^m \dvtx
\forall
j\in\{ 1,\ldots,m \} \exists j'\in\{ 1,\ldots,m \},j'\neq
j,\allowbreak  i_j=i_{j'}\}|$ for $m,n\in\mn$. Then,
\[
C_m^n\leq m! \biggl(\frac{n}{2}
\biggr)^{\lfloor\trup{m}{2}\rfloor}.
\]
\end{lemma}

\begin{pf}
The proof of this lemma is a simple counting exercise. We start with
the case $m\leq2$. One finds easily that $C_1^n=0$ and $C_2^n=n$,
which completes the case $m\leq2$. Next, we consider the case $m>2$.
To this end, we note that $C_m^1=1$, $C_3^2=2$ and $C_m^2\leq2^m\leq
m!$ for $m>3$. This completes the cases $n\leq2$. Now, we do an
induction in $n$. So we let $n\geq2$ and find
\begin{eqnarray*}
C_m^{n+1}&=&C_m^n +
\frac{m(m-1)}{2!}C_{m-2}^n
\\
&&{}+\frac
{m(m-1)(m-2)}{3!}C_{m-3}^n 
+
\cdots+\frac{m(m-1)\cdots3}{(m-2)!}C_{2}^n+1.
\end{eqnarray*}
By induction, this yields
\begin{eqnarray*}
C_m^{n+1}&\leq&m! \biggl[ \biggl(\frac{n}{2}
\biggr)^{\lfloor
\trup{m}{2}\rfloor}+\frac{1}{2!} \biggl(\frac{n}{2}
\biggr)^{\lfloor
\trup{(m-2)}{2}\rfloor}
\\
&&\hphantom{m! \biggl[}{}+\frac{1}{3!} \biggl(\frac{n}{2} \biggr)^{\lfloor\trup
{(m-3)}{2}\rfloor}+\cdots
+\frac{1}{(m-2)!} \biggl(\frac{n}{2} \biggr)^{\lfloor\trup
{2}{2}\rfloor} \biggr] +1.
\end{eqnarray*}
We now assume that $m$ is even. So,
\begin{eqnarray*}
C_m^{n+1}&\leq& m! \biggl[ \biggl(\frac{n}{2}
\biggr)^{\trup
{m}{2}}+\frac
{1}{2!} \biggl(\frac{n}{2}
\biggr)^{\trup{m}{2}-1}+\frac{1}{3!} \biggl(\frac{n}{2}
\biggr)^{\trup{m}{2}-2}+\cdots+\frac{1}{(m-2)!} \biggl(\frac{n}{2} \biggr)
\biggr]+1
\\
&=& m! \Biggl[ \biggl(\frac{n}{2} \biggr)^{\trup{m}{2}}+
\frac
{1}{2!} \biggl(\frac{n}{2} \biggr)^{\trup{m}{2}-1}+\sum
_{j=2}^{\trup
{m}{2}-1} \biggl(\frac{1}{(2j-1)!}+
\frac{1}{ (2j )!} \biggr) \biggl(\frac{n}{2} \biggr)^{\trup{m}{2}-j}
\Biggr]+1
\\
&\leq& m! \Biggl[ \biggl(\frac{n}{2} \biggr)^{\trup{m}{2}}+
\frac
{m}{4} \biggl(\frac{n}{2} \biggr)^{\trup{m}{2}-1}+\sum
_{j=2}^{\trup
{m}{2}-1}{\frac{m}{2}\choose j} \biggl(
\frac{1}{2} \biggr)^j \biggl(\frac{n}{2}
\biggr)^{\trup{m}{2}-j}+ \biggl(\frac{1}{2} \biggr)^{\trup{m}{2}} \Biggr]
\\
&=& m!\sum_{j=0}^{\trup{m}{2}}{\frac{m}{2}
\choose j} \biggl(\frac
{1}{2} \biggr)^j \biggl(
\frac{n}{2} \biggr)^{\trup{m}{2}-j}= m! \biggl( \frac{n+1}{2}
\biggr)^{\lfloor\trup{m}{2}\rfloor}.
\end{eqnarray*}
This completes the proof for $m>2$ with $m$ even. We note finally, that for
odd $m>2$ we have $C_m^n<mC_{m-1}^n\leq m! (\frac{n}{2}
)^{\lfloor\trup{m}{2}\rfloor}$.
\end{pf}

We now settle the bounded part of the problem. Bounded random variables
are in particular subexponential, so one could apply results from
\citep{Viens07}, for
example. But for our purposes, a
direct treatment as in the following is more suitable.
%
%
\begin{lemma}
\label{lemma.helpwoe}
Let $l\in\mn$ and $p,A\geq2$. Then, for the truncation level
$K=\frac{ A}{2}+\sqrt{\frac{A^2}{4}-1}$,
\[
\biggl\llVert \biggl(\max_{1\leq j \leq
N}(\mprn-P)\undz(j)-AM
\frac{\log(N)}{n} \biggr)_+\biggr\rrVert _l\leq \frac{M}{A}+
\frac{lAM}{n}.
\]
\end{lemma}

\begin{pf}
We assume w.l.o.g. $M=1$ and observe that
\[
\me \bigl[\undz_i(j)-\me\undz_i(j)
\bigr]^2\leq\me\undz _i(j)^2\leq1.
\]
Moreover, because of H\"older's and Chebyshev's inequalities and $K\geq1$,
it holds that
\[
\bigl |\undz_i(j)-\me\undz_i(j)\bigr |\leq\bigl |\undz_i(j)\bigr |+\bigl |
\me\undz_i(j)\bigr |\leq K+\frac{1}{K}=A.
\]
These observations, the independence of the random variables and Lemma~\ref{lemma.helpfirst} yield then
\begin{eqnarray*}
&&\me\mathrm{e}^{\trup{n(\mprn-P) \undz(j)}{A}}
\\[-2pt]
&&\quad= \me\mathrm{e}^{\trup{\sum_{i=1}^n(\undz_i(j)-\me\undz
_i(j))}{A}}
\\[-2pt]
&&\quad\leq \biggl(1 +\frac{1}{A^2} \biggr)^n.
\end{eqnarray*}
Next, one checks easily, that the map $x\mapsto\mathrm{e}^{x^{\trup
{1}{l}}}$ is
convex on the set $[(l-1)^{l},\infty)$. Hence, using Jensen's
inequality again, we obtain
\begin{eqnarray*}
&& \biggl\llVert \biggl(\max_{1\leq j \leq
N}(\mprn-P) \undz(j)-A
\frac{\log(N)}{n} \biggr)_+\biggr\rrVert _l
\\[-2pt]
&&\quad\leq \frac{A}{n}\Bigl\llVert \Bigl(\max_{1\leq j \leq
N}n(
\mprn-P) \undz(j)/A-\log(N) \Bigr)_+\vee(l-1)\Bigr\rrVert _l
\\[-2pt]
&&\quad\leq \frac{A}{n}\log \Bigl(\me\exp \Bigl( \Bigl(\max
_{1\leq j
\leq
N}n(\mprn- P) \undz(j)/A-\log(N) \Bigr)_+\vee(l-1) \Bigr)
\Bigr)
\\[-2pt]
&&\quad= \frac{A}{n}\log \Bigl(\me\exp \Bigl( \Bigl(\max
_{1\leq j
\leq
N}n(\mprn- P) \undz(j)/A-\log(N) \Bigr)\vee(l-1) \Bigr)
\Bigr)
\\[-2pt]
&&\quad\leq \frac{A}{n}\log \Bigl(\max_{1\leq j \leq
N}\me\exp
\bigl(n(\mprn- P) \undz(j)/A\bigr)+ \mathrm{e}^{l-1} \Bigr)
\\[-2pt]
&&\quad\leq \frac{A}{n}\log \biggl( \biggl(1+\frac{1}{A^2}
\biggr)^n+\mathrm{e}^{l-1} \biggr).
\end{eqnarray*}
We finally note that $a+b< \mathrm{e}ab$ for all $a,b\geq1$ and find
\begin{eqnarray*}
&&\biggl\llVert \biggl(\max_{1\leq j \leq
N}(\mprn-P) \undz(j)-A
\frac{\log(N)}{n} \biggr)_+\biggr\rrVert _l
\\[-2pt]
&&\quad< \frac{A}{n}\log \biggl( \biggl(1+\frac{1}{A^2}
\biggr)^n\mathrm{e}^{l} \biggr)
\\[-2pt]
&&\quad= \frac{A}{n} \biggl(\log \biggl(1+\frac{1}{A^2}
\biggr)^n+\log \mathrm{e}^{l} \biggr)
\\[-2pt]
&&\quad\leq\frac{1}{A}+\frac{lA}{n}.
\end{eqnarray*}
\upqed
\end{pf}

The results above can now be used to derive a generalization of the
main problem.
%
%
\begin{lemma}
\label{theorem.mainwoe}
Assume that the random variables $Z_i(j)$ are centered. Then, for
$p,A\geq2$, $l\in\mn$, and $p\geq l$,
\[
\label{eq.mainwoe} \biggl\llVert \biggl(Z-AM\frac{\log(2N)}{n} \biggr)_+\biggr
\rrVert _l\leq \biggl(2 \biggl(\frac{2}{A}
\biggr)^{p-1}+(l!)^{\trup{1}{l}}\sqrt{\frac
{2}{n}}+
\frac{1}{A}+\frac{lA}{n} \biggr)M.
\]
\end{lemma}
\begin{pf}
The idea is again to separate the bounded and the unbounded quantities.
The part with the unbounded quantities is treated by elementary
means and Lemma~\ref{lemma.ConIn4.Combi}. For the bounded part, we use
Lemma~\ref{lemma.helpwoe}.

First, we assume w.l.o.g. that $M=1$ and set $K=\frac{
A}{2}+\sqrt{\frac{A^2}{4}-1}$. Then, we deduce with the triangle
inequality that
%
%
\begin{eqnarray}
\label{eq.ConIn4.splitbeg} && \biggl\llVert \biggl(\max_{1\leq j \leq
N}\mprn Z(j)-A
\frac{\log(N)}{
n} \biggr)_+\biggr\rrVert _l
\nonumber
\\
&&\quad= \biggl\llVert \biggl(\max_{1\leq j \leq
N}(\mpr_n-P)Z(j)-A
\frac{\log(N)}{n} \biggr)_+\biggr\rrVert _l
\nonumber
\\
&&\quad\leq\biggl\llVert \biggl(\max_{1\leq j \leq
N}(
\mpr_n-P)\overline{Z}(j)+\max_{1\leq j \leq
N}(
\mpr_n-P)\underline{Z}(j)-A\frac{\log(N)}{
n} \biggr)_+\biggr\rrVert
_l
\\
&&\quad\leq\biggl\llVert \Bigl(\max_{1\leq j \leq
N}(
\mpr_n-P)\overline{Z}(j) \Bigr)_++ \biggl(\max_{1\leq j \leq
N}(
\mpr_n-P)\underline{Z}(j)-A\frac{\log(N)}{
n} \biggr)_+\biggr\rrVert
_l
\nonumber
\\
&&\quad\leq\Bigl\llVert 
\Bigl({ \max_{1\leq j \leq
N}}(
\mpr_n-P)\overline{Z}(j) \Bigr)_+
\Bigr\rrVert
_l+\biggl\llVert 
\biggl({ \max
_{1\leq j \leq
N}}(\mpr_n-P)\underline{Z}(j)-A
\frac{\log(N)}{
n} \biggr)_+
\biggr\rrVert _l.
\nonumber
\end{eqnarray}
So, we are able to treat the unbounded and the bounded quantities
separately. We begin with the unbounded quantities. We first note that
\[
\bigl[(\mpr_n-P)\overline{Z}(j) \bigr]_+^l \leq \bigl((
\mpr_n+P)\overline{\mathcal{E}} \bigr)^l= \bigl((\mpr
_n-P)\overline{\mathcal{E}}+2P\overline{\mathcal{E}}
\bigr)^l.
\]
Hence,
%
%
\begin{equation}
\label{eq.ConIn4.Splitun} \Bigl\llVert \Bigl(\max_{1\leq j \leq
p}(
\mpr_n-P)\overline{Z}(j) \Bigr)_+\Bigr\rrVert _l \leq2P
\overline{\mathcal{E}}+\bigl\llVert (\mpr_n-P)\overline{\mathcal {E}}
\bigr\rrVert _l.
\end{equation}
H\"older's and Chebyshev's inequalities are then used to find
%
%
\begin{equation}
\label{eq.ConIn4.first} P\overline{\mathcal{E}}=\frac{1}{n}\sum
_{i=1}^n\me\overline {\mathcal{E}}_i
\leq\frac{1}{n}\sum_{i=1}^n\bigl({\me
\mathcal{E}_i^p}\bigr)^{\trup
{1}{p}}({\me
1_{\{\mathcal{E}_i>K\}}})^{1-\trup{1}{p}} \leq\frac{1}{K^{p-1}}.
\end{equation}
To bound the left over quantity, we note that for all $i$ and $p\geq
q\in\mn$
\[
\me [\overline{\mathcal{E}}_i-\me\overline{\mathcal
{E}}_i ]^q \leq2^q
\]
so that
\[
\me \bigl[(\overline{\mathcal{E}}_{i_1} - \me\overline{\mathcal
{E}}_{i_1})\cdots(\overline{\mathcal{E}}_{i_l} - \me\overline {
\mathcal{E}}_{i_l}) \bigr] \leq2^l.
\]
Moreover, it holds that
\[
\me \bigl[(\overline{\mathcal{E}}_{i_1} - \me\overline{\mathcal
{E}}_{i_1})\cdots(\overline{\mathcal{E}}_{i_l} - \me\overline {
\mathcal{E}}_{i_l}) \bigr]=0
\]
for all $i_1,\ldots,i_l$ such that there is a $j$ with $i_j\neq
i_{j'}$ for
all $j'\neq j$. With Lemma~\ref{lemma.ConIn4.Combi}, we then get for $n>1$
%
%
\begin{equation}
\label{eq.ConIn4.second} \me \bigl[(\mpr_n-P)\overline{\mathcal{E}}
\bigr]^{l}\leq \frac{2^lC_l^n}{n^l}\leq\frac{2^{l}
l!}{n^l} \biggl(
\frac{n}{2} \biggr)^{\lfloor\trup{l}{2}\rfloor} \leq l!\sqrt{\frac{2}{n}}^l.
\end{equation}
Clearly, this also holds for $n=1$ and $l=1$. For $n=1$ and $l>1$, we
note that
\[
\me \bigl[(\mpr_n-P)\overline{\mathcal{E}} \bigr]^{l}
\leq2^l \leq l! \sqrt2^l,
\]
so that inequality \eqref{eq.ConIn4.second} holds for all $n$ and $l$ under
consideration. Inserting then inequalities \eqref{eq.ConIn4.first} and
\eqref{eq.ConIn4.second} in inequality \eqref{eq.ConIn4.Splitun}, we
obtain the result for the unbounded part
%
%
\begin{equation}
\label{eq.ConIn4.boundunbound} \Bigl\llVert \Bigl(\max_{1\leq j \leq p}(
\mpr_n-P)\overline{Z}(j) \Bigr)_+\Bigr\rrVert _l \leq
\frac{2}{K^{p-1}}+(l!)^{\trup{1}{l}}\sqrt{\frac{2}{n}}.
\end{equation}
Next, we plug the result of Lemma~\ref{lemma.helpwoe} and
inequality \eqref{eq.ConIn4.boundunbound} in
inequality \eqref{eq.ConIn4.splitbeg} to derive
\[
\biggl\llVert \biggl( \max_{1\leq j \leq
N}\mprn Z(j)-A\frac{\log(N)}{n}
\biggr)_+\biggr\rrVert _l \leq2 \biggl(\frac{2}{A}
\biggr)^{p-1}+(l!)^{\trup{1}{l}}\sqrt {\frac
{2}{n}}+
\frac{1}{A}+\frac{lA}{n}.
\]
Finally, we define $Z(j+N):=-Z(j)$ for $1\leq j \leq N$. We
then get
\begin{eqnarray*}
\biggl\llVert \biggl(Z-A\frac{\log(2N)}{n} \biggr)_+\biggr\rrVert
_l&=&\biggl\llVert \biggl( \max_{1\leq j \leq
2N}\mprn Z(j)-A
\frac{\log(2N)}{n} \biggr)_+\biggr\rrVert _l
\\
&\leq&2 \biggl(\frac{2}{A} \biggr)^{p-1}+(l!)^{\trup{1}{l}}
\sqrt {\frac
{2}{n}}+\frac{1}{A}+\frac{lA}{n}
\end{eqnarray*}
replacing $N$ by $2N$ in the results above.
\end{pf}

Theorem~\ref{sec4.Corrollary1} is now a simple corollary.
\begin{pf*}{Proof of Theorem~\ref{sec4.Corrollary1}}
Set $A=2\sqrt n$ in Lemma~\ref{theorem.mainwoe}.
\end{pf*}

\section*{Acknowledgements}

We acknowledge partial financial support as
members of the
German-Swiss research group FOR916 (Regularization and
Qualitative Constraints) with Grant number 20PA20E-134495/1. We also thank
Micha\"el Chichignoud and Mohamed Hebiri for their valuable comments. Finally,
we thank the editor and the referees for their helpful suggestions.


%

\printhistory

\end{document}